\documentclass[11pt]{article}

\usepackage[T1]{fontenc}
\usepackage[utf8]{inputenc}
\usepackage{lmodern}
\usepackage{amsfonts,amssymb,amsmath,amsthm,mathrsfs,mathtools}
\usepackage[dvipsnames]{xcolor}
\usepackage[margin=1.25in]{geometry}
\usepackage{footnote}

\usepackage[colorlinks=true]{hyperref}
\hypersetup{
  pdftitle={A higher-dimensional Schrödinger maximal estimate along families of bilipschitz curves},
  pdfauthor={Javier Minguillón},
  pdfcreator={LaTeX with hyperref},
  pdfproducer={pdfLaTeX},
  colorlinks=true,
  linkcolor={blue!80!black},
  citecolor={blue!50!black},
  urlcolor={blue!50!black}
}
\usepackage{cleveref}

\newtheorem{theorem}{Theorem}
\newtheorem{lemma}[theorem]{Lemma}
\newtheorem{proposition}[theorem]{Proposition}
\newtheorem{corollary}[theorem]{Corollary}
\theoremstyle{remark}
\newtheorem{remark}[theorem]{Remark}

\newcommand{\R}{\mathbb{R}}
\newcommand{\Z}{\mathbb{Z}}

\newcommand{\eps}{\varepsilon}
\newcommand{\abs}[1]{\lvert #1\rvert}
\newcommand{\norm}[1]{\lVert #1\rVert}
\newcommand{\Mtheta}{M_{\Theta}}
\newcommand{\dimM}{\dim_{M}}
\newcommand{\supp}{\operatorname{supp}}

\title{A higher-dimensional Schrödinger maximal estimate along families of bi-Lipschitz curves}

\date{\today}
\author{Javier Minguillón}
\begin{document}

\maketitle

\begin{abstract}
Let $\Theta\subset\R^n$, $n\geq 1$ be compact with upper Minkowski dimension $\beta$, and let $\gamma(x,t,\theta)$ be a family of curves which is bi-Lipschitz in the first variable and Lipschitz in the last two. We prove the maximal estimate
$$\left\|\sup_{(\tau,\theta)\in[0,1)\times\Theta}\abs{e^{i\tau\Delta}f(\gamma(\,\cdot\,,\tau,\theta))}\right\|_{L^2(B^n_1)}\lesssim_s\norm{f}_{H^s(\R^n)},$$
for every $s>\tfrac{n(\beta+1)}{2(n+1)}$, in the range $0\le\beta\le n-1$. In the range $n-1<\beta\leq n$ we prove the threshold $s>\min\big\lbrace\frac\beta 2 + \frac 1{2(n+1)},\frac n2\big\rbrace$.
The proof relies on a fractal $L^2$ estimate of Du--Zhang. We also use the temporal localization of Lee, the locally constant property, and a counting argument allowed by the Lipschitz properties of the curve.
\end{abstract}

\vspace{-10mm}
\section{Introduction and main result}\label{sec:intro}

Consider the linear Schrödinger equation on $ \R^n\times \R $, $ n\geq 1, $ given by
\begin{equation}
	\label{equation_schrodinger with boundary datum}
	\begin{cases}
		i\partial_t u(x,t) + \Delta_x u (x,t) = 0,
		\\
		u(x,0) = f(x).
	\end{cases}
\end{equation}
Its solution can be formally expressed as
\begin{equation}
	e^{it\Delta}f(x) = \int_{\R^n} e^{2\pi i x\cdot \xi} e^{-4\pi^2 it|\xi|^2} \widehat f(\xi) d\xi.
\end{equation}

In 1980 Carleson \cite{Car} asked for the minimal regularity $s$ such that $e^{it\Delta}f\to f$ almost everywhere as $t\to0^+$ for every $f\in H^s(\R^n)$. For $n=1$ the sufficiency of $s\ge\tfrac14$ was found by Carleson \cite{Car} and Dahlberg--Kenig \cite{DK} proved the necessity. Over the years, many researchers have contributed to this problem, including Carbery, Cowling, Sjölin, Vega, Moyua, Vargas, Tao, and Bourgain, among others. In higher dimensions the problem stayed open for three decades, and was recently solved up to the endpoint. Bourgain \cite{Bou} constructed initial data in $H^s$ with $s<\tfrac{n}{2(n+1)}$ whose solutions diverge on a set of positive measure. Soon after, Du--Guth--Li \cite{DGL} and Du--Zhang \cite{DZ} proved that $s>\tfrac{n}{2(n+1)}$ suffices in dimension $n=2$ and $n\geq 3$ respectively. The second article does it via the fractal $L^2$ estimate stated in Proposition~\ref{prop:DZ} below.

Convergence along curves
is a strictly harder question. If the approach direction $\theta$ is restricted to a fixed compact set $\Theta$, the answer should interpolate between two extremes. First, for a single direction one expects the threshold $s\geq \frac n{2(n+1)}$. Second, for a full cone of directions ($\Theta$ of positive measure) Sjögren--Sjölin \cite{SS} showed that $s>\tfrac n2$ is necessary. The condition $s>\frac n2$ is known to be sufficient by Sobolev embedding. 
The upper Minkowski dimension $\beta$ of $\Theta$ is the parameter that interpolates between these two extremes.
In $\R^1$ this interpolation was carried out by Cho--Lee--Vargas \cite{CLV}, who obtained the threshold $s>\tfrac{\beta+1}{4}$ (see also Shiraki \cite{Shi} for the fractional Schrödinger operator). In $\R^2$, for a family of curves rather than straight directions, this was studied by Li--Wang--Yan \cite{LWY}, who obtained $s>\tfrac{\beta+1}{3}$. The purpose of this note is to prove the analogous estimate in all dimensions, for the class of curves considered by Li--Wang--Yan.

Let $n\ge1$ throughout this paper. Let $\Theta\subset\R^n$ be a compact set. Denote its upper Minkowski dimension by
\begin{equation}\label{eq:dimM}
\beta:=\dimM(\Theta)=\limsup_{\delta\to0^+}\frac{\log N(\Theta,\delta)}{\log 1/\delta}\in[0,n],
\end{equation}
where $N(E,\delta)$ denotes the minimum number of $\delta$-balls needed to cover $E$.

Let $\gamma:\R^n\times[0,1)\times\Theta\to\R^n$ with $\gamma(x,0,\theta)=x$ be a family of curves satisfying the conditions of Li--Wang--Yan \cite{LWY}. That is, for some constants $C_1\ge1$ and $C_2,C_3>0$, and for all $x,x'\in\R^n$, $t,t'\in[0,1)$ and $\theta,\theta'\in\Theta$,
\begin{align}
C_1^{-1}\abs{x-x'}\ \le\ \abs{\gamma(x,t,\theta)-\gamma(x',t,\theta)}\ &\le\ C_1\abs{x-x'},\tag{C1}\label{C1}\\
\abs{\gamma(x,t,\theta)-\gamma(x,t',\theta)}\ &\le\ C_2\abs{t-t'},\tag{C2}\label{C2}\\
\abs{\gamma(x,t,\theta)-\gamma(x,t,\theta')}\ &\le\ C_3\abs{\theta-\theta'}.\tag{C3}\label{C3}
\end{align}
We call $x$ the \emph{base point} of the curve $t\mapsto\gamma(x,t,\theta)$. Condition \eqref{C1} says that at each fixed time the map $x\mapsto\gamma(x,t,\theta)$ is a bi-Lipschitz deformation of $\R^n$. Denote 
\begin{equation}
\Gamma_{x,t} := \{\gamma(x,t,\theta):\theta\in \Theta\}.
\end{equation}
Define the maximal operator
\begin{equation}\label{eq:Mtheta}
\Mtheta f(x):=\sup_{(\tau,\theta)\in[0,1)\times\Theta}\abs{e^{i\tau\Delta}f(\gamma(x,\tau,\theta))}.
\end{equation}

\subsection*{Notation}
In this article, $B^n_r(x)\subset\R^n$ is the open ball of radius $r>0$ centered at $x\in\R^n$, and we write $B^n_r:=B^n_r(0).$  
For $r>0$, we let $Q^n_r(v)$ denote the closed cube of side $r$ centered at $v\in\R^n$. A \emph{lattice unit cube} is a cube $Q^n_1(v)$ with $v\in\Z^n$, or a cube $Q^{n+1}_1(w)$ with $w\in\Z^{n+1}$. Let $\pi:\R^{n+1}\to\R^n$ be the projection to the first $n$ variables. For $\lambda\ge1$ set $A(\lambda):=\{\xi\in\R^n:\frac\lambda2\le\abs\xi\le2\lambda\}$. We write $A\lesssim B$ for $A\le CB$ with $C$ an absolute constant allowed to depend on $n$, $C_1$, $C_2$, $C_3$ and $\Theta$. We use $A\lesssim_\eps B$ when $C=C_\eps$ may in addition depend on $\eps$. We denote $A\sim B$ for $A\lesssim B\lesssim A$. Throughout, $R\gg1$ denotes a large number and $N$ some integer larger than $100n$.

We are going to use the fractal $L^2$ estimate of Du--Zhang in the following form. For a union $X=\bigcup_kQ_k$ of lattice unit cubes contained in $B^{n+1}_{R}(0)$ and for $1\le\alpha\le n+1$, define the fractal density
\begin{equation}\label{eq:gammadef}
\nu_{X,\alpha}\ :=\ \max_{\substack{B^{n+1}_r(w)\subseteq B^{n+1}_{R}(0)\\ r\ge1,\ \abs w\le R}}\ \frac{\#\{k:\ Q_k\subset B^{n+1}_r(w)\}}{r^{\alpha}} .
\end{equation}

\begin{proposition}[Du--Zhang {\cite[Corollary~1.7]{DZ}}]\label{prop:DZ}
Let $R\ge1$, let $X$ and $\nu_{X,\alpha}$ be as above, and let $h\in L^2(\R^n)$ with $\supp\widehat h\subset B^n_2(0)$. Then, for every $\eps>0$,
\begin{equation}\label{eq:DZ}
\norm{e^{it\Delta}h}_{L^2(X)}\ \lesssim_\eps\ \nu_{X,\alpha}^{\frac1{n+1}}\,(R)^{\frac{\alpha}{2(n+1)}+\eps}\,\norm h_2 .
\end{equation}
\end{proposition}
Our main result is the following.
\begin{theorem}\label{thm:main}
Given $\Theta\subset \R^n$ compact, $\beta=\dimM(\Theta)$ and $\gamma:\R^n\times[0,1)\times\Theta\to\R^n$, with $\gamma(x,0,\theta)=x$, satisfying \eqref{C1},\eqref{C2} and \eqref{C3} we have
\begin{equation}\label{eq:main}
\norm{\Mtheta f}_{L^2(B^n_1)}\ \lesssim_s\ \norm{f}_{H^s(\R^n)},
\end{equation}
for all $f\in H^s(\R^n)$ whenever
\begin{equation}\label{eq:thresholds}
s\ >\sigma:= \
\begin{cases}
	\dfrac{n(\beta+1)}{2(n+1)}, & 0\le\beta\le n-1,\\[2mm]
	\dfrac{\beta}{2}+\dfrac{1}{2(n+1)}, & n-1<\beta\le n-\dfrac{1}{n+1},\\[2mm]
	\dfrac n2, & n-\dfrac{1}{n+1}<\beta\le n.
\end{cases}
\end{equation}
\end{theorem}

By a standard argument, this yields a convergence result under the same hypotheses.

\begin{corollary}
Given $\Theta\subset \R^n$ compact, $\beta=\dimM(\Theta)$ and $\gamma:\R^n\times[0,1)\times\Theta\to\R^n$, with $\gamma(x,0,\theta)=x$, satisfying \eqref{C1},\eqref{C2} and \eqref{C3}, we have
\begin{equation}\label{eq:aeconverg}
\lim _{\substack{(y, t) \rightarrow(x, 0) \\ y \in \Gamma_{x,t},\ t>0}} e^{i t \Delta} f(y)=f(x) \;\text { a.e. } x\in\R^n
\end{equation}
for all $f\in H^s(\R^n)$ with $s>\sigma$ and $\sigma$ as defined in \eqref{eq:thresholds}.
\end{corollary}

The threshold $\sigma$ is the one the method produces when Proposition~\ref{prop:DZ} is applied at the natural exponent $\alpha=n-\beta$ . This requires $\alpha\ge1$, that is $\beta\le n-1$. The second range is a special case, in which $\alpha=n-\beta<1$ is not a valid choice to apply Proposition \ref{prop:DZ}. Here we substitute the smallest admissible value $\alpha=1$. The third range is the trivial threshold given by Sobolev embedding. Notice that, at $\beta=n-1$ we have $\sigma=\tfrac{n^2}{2(n+1)}$, and $\sigma = \frac n2$ at $\beta= n-\tfrac1{n+1}$, so the exponent in Theorem~\ref{thm:main} is a continuous, increasing function of $\beta$ on $[0,n]$.

\begin{remark}[Comments on the exponent and known cases]\label{rem:sanity}
The threshold $\sigma$ matches some cases where an answer is independently known:
\begin{itemize}
\item[(i)] At $n=1$, the best known value is $\tfrac{\beta+1}{4}$, the theorem of Cho--Lee--Vargas \cite{CLV} (see also Shiraki \cite{Shi} for the fractional Schrödinger operator), conjectured sharp for $\beta\in(0,1)$ and known to be sharp for $\beta\in\{0,1\}$. Note that $\tfrac{n(\beta+1)}{2(n+1)} = \tfrac{\beta+1}{4}$ at $n=1$. However our $\sigma$ only takes that value at $\beta\in\{0,1\}$.
\item[(ii)] At $\beta=0$, when there is a single curve, $\sigma=\tfrac{n}{2(n+1)}$ is the sharp threshold for the almost everywhere convergence problem, by Du--Guth--Li \cite{DGL}, Du--Zhang \cite{DZ} and the counterexample of Bourgain \cite{Bou}.
\item[(iii)] At $n=2$, the best known value is $\tfrac{\beta+1}{3}$ obtained by Li--Wang--Yan \cite{LWY}. Our $\sigma =\tfrac{n(\beta+1)}{2(n+1)}= \tfrac{\beta+1}{3}$ for all $\beta\in[0,1]\cup\{2\}$ and $\sigma>\tfrac{\beta+1}{3}$ for all $\beta\in(1,2)$.
\item[(iv)] At $\beta=n$, the sufficiency threshold $s>\tfrac n2$ is given by Sobolev embedding. Convergence on a cone of directions requires $s> \frac n2$ by Sjögren--Sjölin \cite{SS}. Our argument in Section \refeq{sec:highbeta} allows us to reach the sufficiency of $s>\tfrac n2+\tfrac1{2(n+1)}$, not lower.
\end{itemize}
\end{remark}

The threshold $\sigma$ improves known thresholds for $n\geq 3$. Consider the 
maximal estimate
\begin{equation}\label{eq:fixedparam}
\Bigl\|\sup_{0<t<1}\abs{e^{it\Delta}f}\Bigr\|_{L^p(B^n_1)}\ \le\ C_s\norm f_{H^s(\R^n)} .
\end{equation}
Li--Wang--Yan prove two things about it. First, if \eqref{eq:fixedparam} holds for some $p>1$ and every $s>\tfrac{n}{2(n+1)}$, then
\begin{equation}\label{eq:LWYcurve}
\norm{\Mtheta f}_{L^p(B^n_1)}\ \le\ C_s\norm f_{H^s(\R^n)},
\end{equation}
holds for every $s>\tfrac\beta p+\tfrac{n}{2(n+1)}$. This is their Theorem 1.4. 
Second, if such a $p$ satisfies \eqref{eq:fixedparam}, then necessarily
\begin{equation}\label{eq:pbound}
p\ \le\ \frac{2(n+1)}{n} ,
\end{equation}
which is their Theorem 1.5. It is known that \eqref{eq:fixedparam} holds for $s>\frac{n}{2(n+1)}$ at $p=\tfrac{2(n+1)}{n}$ for $n=1$ from Kenig--Ponce--Vega \cite{KPV} and $n=2$ from Du--Guth--Li \cite{DGL}. For $n\geq 3$, Du--Kim--Wang--Zhang \cite{DKWZ}
proved that \eqref{eq:fixedparam} fails for some $s>\frac n{2(n+1)}$ whenever $p$ belongs to a certain range containing
\begin{equation}\label{eq:lowerp}
	2+\frac4{(n-1)(n+2)}< p \leq \frac{2(n+1)}{n}.
\end{equation}

In the range $0\le\beta\le n-1$, our sufficiency threshold for convergence \eqref{eq:aeconverg} to hold is exactly $\sigma=\tfrac{n}{2(n+1)}+\tfrac{\beta}{p}$ at the extreme value $p=\tfrac{2(n+1)}{n}$ allowed by \eqref{eq:pbound}. The proof below never uses \eqref{eq:fixedparam}. We obtain unconditionally, the best possible regularity threshold that \eqref{eq:pbound} would allow in $n=1$ and $n=2$. For $n\geq3$, we obtain a threshold that Theorem 1.4 of Li--Wang--Yan cannot obtain because of \eqref{eq:lowerp}. For $n\geq 2$, at $p=2$, \eqref{eq:fixedparam} is known to hold for all $s>\frac{n}{2(n+1)}$ from \cite{DGL} and \cite{DZ}. This gives a sufficiency threshold for \eqref{eq:aeconverg} of $ \frac{n}{2(n+1)} +\frac{\beta}{2} > \sigma$ for $\beta>0$.

In the remaining range $n-1<\beta< n$ we do not reach $\tfrac{n(\beta+1)}{2(n+1)}$. 

\section{The locally constant property}\label{sec:lcp}

 We use the Locally Constant Property in the following form.

\begin{proposition}[Locally constant property]\label{prop:lcp}
Let $c\ge1$ and $K\ge1$, and let $f\in L^2(\R^n)$ with $\supp\widehat f\subset B^n_c(0)$. For $\mathfrak l\in\Z^n$ write
\begin{equation}\label{eq:htrans}
\widehat{f_{\mathfrak l}}(\xi)=e^{2\pi i\mathfrak l\cdot\xi}\,\widehat f(\xi).
\end{equation}
 Let $(x,t)\in\R^{n+1}$ and let $\mathcal Q\subset\R^{n+1}$ be a unit cube such that $\abs{(x,t)-(x',t')}\le K$ for all $(x',t')\in\mathcal Q$. Then, for every $p\ge1$,
\begin{equation}\label{eq:lcp-cubes}
\bigl|e^{it\Delta}f(x)\bigr|^{p}\ 
\lesssim\
 \sum_{\mathfrak l\in\Z^{n}}\frac{1}{(1+\abs{\mathfrak l})^{n+1}}\int_{\mathcal Q}\bigl|e^{is\Delta}f_{\mathfrak l}(y)\bigr|^{p}\,dy\,ds ,
\end{equation}
where the implicit constant may depend on $c$ and $K$.
\end{proposition}

This is standard, and we do not prove it. It is Proposition 3.2 of \cite{MSV}. More general versions of this principle are Lemma~2.1 in Tao \cite{Tao} and Lemma~2.10 in Christ \cite{Chr}.


We will not use Proposition~\ref{prop:lcp} directly, but rather the following consequence.

\begin{corollary}\label{cor:lcp}
Let $h\in L^2(\R^n)$ with $\supp \hat h \subset A(1)+B^n_{1/100}$. Let $E\subset\R^{n+1}$ be non-empty. Then there is a lattice unit cube $\mathcal Q\subset\R^{n+1}$, containing a point of $E$, such that
\begin{equation}\label{eq:lcp-sup}
\sup_{(z,t)\in E}\abs{e^{it\Delta}h(z)}^2\ \lesssim\ \sum_{\mathfrak l\in\Z^{n}}\frac{1}{(1+\abs{\mathfrak l})^{n+1}}\,\norm{e^{i\cdot\Delta}h_{\mathfrak l}}_{L^2(\mathcal Q)}^2 .
\end{equation}
\end{corollary}

\begin{proof}
We observe that $e^{it\Delta}h(\cdot)$ is bounded. Indeed, since $\widehat h$ is supported on a bounded set, Cauchy--Schwarz and Plancherel give
\begin{equation}\label{eq:L1L2}
	\norm{\widehat h}_{1}\ \le\ \abs{\supp\widehat h}^{1/2}\,\norm{\widehat h}_{2}\ \lesssim\ \norm h_2 ,
\end{equation}
so $\widehat h\in L^1$ and $\abs{e^{it\Delta}h(\cdot)}\lesssim\norm{\widehat h}_{L^1}\lesssim\norm h_2$ uniformly in $t$.
	
Since $e^{it\Delta}h(\cdot)$ is bounded, the supremum in \eqref{eq:lcp-sup} is finite, so we may pick $a_*=(x_*,t_*)\in E$ with
\begin{equation}\label{eq:halfsup}
\bigl|e^{it_*\Delta}h(x_*)\bigr|^2\ \ge\ \dfrac12\sup_{(z,t)\in E}\abs{e^{it\Delta}h(z)}^2.
\end{equation}
Let $\mathcal Q$ be a lattice unit cube containing $a_*$. Every point of $\mathcal Q$ lies at distance $\sqrt{n+1}$ or less from $a_*$. Notice that $A(1)+B^n_{1/100}\subset B^n_{2}+B^n_{1/100}$. We apply Proposition~\ref{prop:lcp} with $p=2$, $c=2+\frac1{100}$, and $K=\sqrt{n+1}$, and obtain \eqref{eq:lcp-sup}.
\end{proof}

\section{Reductions}\label{sec:red}
In this section, we reduce \eqref{eq:main} to a simpler statement using the Lipschitz structure of the curve and the standard argument. That is, we use Littlewood--Paley, parabolic rescaling and temporal localization. Then, we perform a further reduction related to the specific curves we are considering.
\subsection{Littlewood--Paley}\label{ssec:LP}
Fix $\psi\in C^\infty_c(B^n_1)$ with $\psi\equiv1$ on $B^n_{1/2}$. Set $\psi_j:=\psi(2^{-j}\cdot)$ and $\varphi_{j+1}:=\psi_{j+1}-\psi_j$ for $j\ge0$, so that $\supp\varphi_j\subset A(2^{j-1})$ and $\psi_0+\sum_{j\ge1}\varphi_j\equiv1$. Consider the following inequality. For every $\eps>0$ there exists $C_\eps>0$ such that for every $R\gg1$ and every $g$ with $\supp\widehat g\subset A(R)$,
\begin{equation}\label{eq:single}
\norm{\Mtheta g}_{L^2(B^n_1)}\ \le\ C_\eps\,R^{\sigma+\eps}\,\norm g_2 .
\end{equation}

Let us prove that \eqref{eq:main} is a consequence of \eqref{eq:single}. First, $\norm{\Mtheta(\check\psi_0*f)}_{L^2(B^n_1)}\lesssim\norm{\widehat{f}\,\psi_0}_{L^1}\lesssim\norm f_2$. Thus, using the sublinearity of $\Mtheta$ and then Cauchy--Schwarz,
\begin{align}
\norm{\Mtheta f}_{L^2(B^n_1)}^2
&\ \lesssim\ \|f\|_2^2 + \Bigl(\sum_{j\geq 1}\norm{\Mtheta(\check\varphi_j*f)}_{L^2(B^n_1)}\Bigr)^{\!2}
\label{eq:LP}\\
&\ \lesssim\ \|f\|_2^2 + \Bigl(\sum_{j\geq 1}2^{-2j\eps}\Bigr)\sum_{j\geq 1}2^{2j\eps}\norm{\Mtheta(\check\varphi_j*f)}_{L^2(B^n_1)}^2,
\notag \\ 
\intertext{and now, applying \eqref{eq:single} at $R=2^{j-1}$ to each summand,}
&\ \lesssim_\eps\ \|f\|_2^2 + \sum_{j\geq 1}2^{2j(\sigma+2\eps)}\norm{\check\varphi_j*f}_2^2\notag
\sim\ \norm f^2_{H^{\sigma+2\eps}}.\notag
\end{align}

\subsection{Parabolic rescaling}\label{ssec:rescale}
Fix $R\gg1$ and $g$ with $\supp\widehat g\subset A(R)$. Write
\begin{equation}\label{eq:prerescale}
\norm{\Mtheta g}_{L^2(B^n_1)}
=\Bigl(\int_{B^n_1}\ \sup_{(\tau,\theta)\in[0,1)\times\Theta}\abs{e^{i\tau\Delta}g(\gamma(x,\tau,\theta))}^2\,dx\Bigr)^{1/2}.
\end{equation}
Change variables $y=Rx$ and $t'=R^2\tau$, 
and define $h:=R^{-n/2}g(R^{-1}\cdot)$, so that $\widehat h=R^{n/2}\widehat g(R\,\cdot)$ has support in $A(1)$ and $\norm h_2=\norm g_2$. Let us give a name to the rescaled curve that arises,
\begin{equation}\label{eq:gammaR}
\gamma_R(y,t',\theta):=R\,\gamma\bigl(R^{-1}y,\ R^{-2}t',\ \theta\bigr).
\end{equation}
Now, $R\,\gamma(x,\tau,\theta)=\gamma_R(y,t',\theta)$ and $\gamma_R(y,0,\theta)=y$. By \eqref{C1}--\eqref{C3}, the rescaled curve satisfies, uniformly in $R$,
\begin{align}
C_1^{-1}\abs{y-y'}\ \le\ \abs{\gamma_R(y,t',\theta)-\gamma_R(y',t',\theta)}\ &\le\ C_1\abs{y-y'},\tag{C1$_R$}\label{C1R}\\
\abs{\gamma_R(y,t',\theta)-\gamma_R(y,t'',\theta)}\ &\le\ C_2R^{-1}\abs{t'-t''},\tag{C2$_R$}\label{C2R}\\
\abs{\gamma_R(y,t',\theta)-\gamma_R(y,t',\theta')}\ &\le\ C_3R\,\abs{\theta-\theta'}.\tag{C3$_R$}\label{C3R}
\end{align}
The change of variables yields
\begin{align}\label{eq:rescaled-explicit}
	\norm{\Mtheta g}_{L^2(B^n_1)}
	&=
	R^{-n/2}\Bigl(\int_{B^n_R}\ \sup_{(t',\theta)\in[0,R^2)\times\Theta}
	\bigl|\bigl(e^{iR^{-2}t'\Delta}g\bigr)\bigl(\gamma(R^{-1}y,\,R^{-2}t',\,\theta)\bigr)\bigr|^2\,dy\Bigr)^{1/2}.
	\intertext{By definition of $h$ and $\gamma_R$, the term inside the absolute value is $e^{iR^{-2}t'\Delta}g(R^{-1}\gamma_R(y,t',\theta))= R^{n/2}e^{it'\Delta}h(\gamma_R(y,t',\theta))$, thus}
	&=\Bigl(\int_{B^n_R}\ \sup_{(t',\theta)\in[0,R^2)\times\Theta}\abs{e^{it'\Delta}h(\gamma_R(y,t',\theta))}^2\,dy\Bigr)^{1/2}.
	\notag
\end{align}
By \eqref{eq:rescaled-explicit}, bound \eqref{eq:single} is equivalent to
\begin{equation}\label{eq:global}
\Bigl\|\sup_{(t',\theta)\in[0,R^2)\times\Theta}\abs{e^{it'\Delta}h(\gamma_R(y,t',\theta))}\Bigr\|_{L^2(B^n_R,\,dy)}\ \le\ C_\eps\,R^{\sigma+\eps}\,\norm h_2 ,
\end{equation}
for all $h$ with $\supp\widehat h\subset A(1)$.

\subsection{Temporal localization}\label{ssec:temporal}
Estimate \eqref{eq:global} has time ranging over $[0,R^2)$. As in Lee \cite{Lee}, it reduces to the following family of estimates with time ranging over intervals of length $R$. For every $\eps>0$ there is $C_\eps>0$ such that for every $\tilde t\in[0,R^2)$ and every $h$ with $\supp\widehat h\subset A(1)+B^n_{1/100}$,
\begin{equation}\label{eq:short}
\Bigl\|\sup_{(t,\theta)\in[0,R)\times\Theta}\abs{e^{it\Delta}h\bigl(\gamma_R(y,\tilde t+t,\theta)\bigr)}\Bigr\|_{L^2(B^n_R,\,dy)}\ \le\ C_\eps\,R^{\sigma+\eps}\,\norm h_2.
\end{equation}

The reduction relies on the following lemma of Lee, which we state in the form in which we use it.

\begin{lemma}[Lee {\cite[Lemma~2.1]{Lee}}]\label{lem:lee}
There is a constant $c_1=c_1(n)\ge1$ with the following property. Let $R\gg1$, let $M\ge1$, and partition $[0,R^2)$ into $R$ many $I_j:=[jR,(j+1)R)$, $0\le j<R$. Let $h\in L^2(\R^n)$ with $\supp\widehat h\subset A(1)$ and let $\eps>0$. Then there are functions $h_j\in L^2(\R^n)$, $0\le j<R$, whose Fourier supports satisfy
\begin{equation}\label{eq:lee-supp}
\supp\widehat{h_j}\ \subset\ \supp\widehat h+B^n_{c_1R^{-1}} ,
\end{equation}
such that
\begin{equation}\label{eq:lee-props}
\Bigl(\sum_{j=0}^{R-1}\norm{h_j}_2^2\Bigr)^{1/2}\ \le\ C_\eps R^\eps\norm h_2 ,
\end{equation}
and such that, for every $N>0$,
\begin{equation}\label{eq:lee-point}
\abs{e^{it\Delta}h(z)}\ \le\ \abs{e^{it\Delta}h_j(z)}+C_NR^{-N}\norm h_2,
\end{equation}
for all $z\in Q^n_{MR}(0)$ and all $t\in I_j$.
\end{lemma}

We now prove the reduction. 
\begin{lemma}\label{lem:temporal}
If \eqref{eq:short} holds for every $\tilde t\in[0,R^2)$, then \eqref{eq:global} holds.
\end{lemma}

\begin{proof}
Set $t_j:=jR$ and partition $[0,R^2)$ into $R$ many intervals $I_j:=[t_j,t_{j+1})$, $0\le j<R$. Let $M:=2C_2+2$. Apply Lemma~\ref{lem:lee} with this $M$ and with the partition $\{I_j\}$. We obtain functions $h_j$, $0\leq j < R$ associated to $h$.
Since 
$R\gg100c_1$, by
\eqref{eq:lee-supp},
\begin{equation}\label{eq:lee-supp-here}
\supp\widehat{h_j}\ \subset\ \supp\widehat h+B^n_{c_1R^{-1}}\ \subset\ A(1)+B^n_{1/100}.
\end{equation}
Observe that, for $y\in B^n_R$, writing $\gamma_R(y,t',\theta)-y=\gamma_R(y,t',\theta)-\gamma_R(y,0,\theta)$ and using \eqref{C2R},
\begin{equation}\label{eq:inQ}
\abs{\gamma_R(y,t',\theta)}\le\abs y+C_2R^{-1}t'\le(1+C_2)R,
\end{equation}
for all $t'\in[0,R^2),\ \theta\in\Theta$.
Thus, $\gamma_R(y,t',\theta)\in Q^n_{MR}(0)$ and \eqref{eq:lee-point} holds for $z=\gamma_R(y,t',\theta)$. Writing $t'=t_j+t$ with $t\in[0,R)$, define
\begin{equation}\label{eq:boost}
G_j:=e^{it_j\Delta}h_j ,
\end{equation}
Observe that
$\norm{G_j}_2=\norm{h_j}_2$ and $\supp\widehat{G_j}=\supp\widehat{h_j}\subset A(1)+B^n_{1/100}$. 
Let us prove \eqref{eq:global}. Write
\begin{align}
&\phantom{\le}\Bigl\|\sup_{(t',\theta)\in[0,R^2)\times\Theta}\abs{e^{it'\Delta}h(\gamma_R(y,t',\theta))}\Bigr\|_{L^2(B^n_R,\,dy)}\ \Bigl\|
\\&\ =\sup_{(t,\theta)\in[0,R)\times\Theta}\ \sup_{0\leq j < R}\ \abs{e^{i(t_j+t)\Delta}h(\gamma_R(y,t_j+t,\theta))}\Bigr\|_{L^2(B^n_R)}^2,
\label{eq:tempchain}\\
\intertext{applying \eqref{eq:lee-point}, \eqref{eq:inQ}, and \eqref{eq:boost},}
&\ \le\ \Bigl\|\sup_{(t,\theta)\in[0,R)\times\Theta}\ \sup_{0\leq j < R}\ \abs{e^{it\Delta}G_j(\gamma_R(y,t_j+t,\theta))}\Bigr\|_{L^2(B^n_R)}^2+C_{N}R^{-N}\norm h_2^2,\notag\\
\intertext{majorizing the supremum in $j$ by a sum in $j$,}
&\ \le\ \sum_{0\le j<R}\Bigl\|\sup_{(t,\theta)\in[0,R)\times\Theta}\abs{e^{it\Delta}G_j(\gamma_R(y,t_j+t,\theta))}\Bigr\|_{L^2(B^n_R)}^2
+C_{N}R^{-N}\norm h_2^2,\notag\\
\intertext{applying \eqref{eq:short} with $\tilde t=t_j$ to each $G_j$,}
&\ \le\ C_\eps^2R^{2\sigma+2\eps}\sum_{0\le j<R}\norm{h_j}_2^2+C_NR^{-N}\norm h_2^2,\notag\\
\intertext{and then \eqref{eq:lee-props},}
&\ \le\ C'_\eps R^{2\sigma+4\eps}\norm h_2^2,\notag
\end{align}
which is the square of the right-hand side of \eqref{eq:global} after renaming $\eps$.
\end{proof}

\subsection{Passing to the evolved sets}\label{ssec:evolved}
Fix $\tilde t\in[0,R^2)$ and define, for each $y\in\R^n$, the \emph{evolved set}
\begin{equation}\label{eq:evolved}
\mathcal E_y(\tilde t):=\{\gamma_R(y,\tilde t,\theta):\theta\in\Theta\}\subset\R^n .
\end{equation}
Think of this set as the evolution of $\Theta$ by the curve, with the base point $y$ fixed, at time $\tilde t$. By \eqref{C2R}, for $t\in[0,R)$,
\begin{equation}\label{eq:drift}
\abs{\gamma_R(y,\tilde t+t,\theta)-\gamma_R(y,\tilde t,\theta)}\ \le\ C_2R^{-1}t\ \le\ C_2 .
\end{equation}
Hence
\begin{equation}\label{eq:tosets}
\sup_{(t,\theta)\in[0,R)\times\Theta}\abs{e^{it\Delta}h(\gamma_R(y,\tilde t+t,\theta))}
\ \le\ \sup_{(z,t)\in(\mathcal E_y(\tilde t)+B^n_{C_2})\times[0,R)}\abs{e^{it\Delta}h(z)} ,
\end{equation}
and \eqref{eq:short} reduces to proving that, for all $h$ with $\supp\widehat h\subset A(1)+B^n_{1/100}$,
\begin{equation}\label{eq:slice}
\Bigl\|\sup_{(z,t)\in(\mathcal E_y(\tilde t)+B^n_{C_2})\times[0,R)}\abs{e^{it\Delta}h(z)}\Bigr\|_{L^2(B^n_R,\,dy)}\ \le\ C_\eps\,R^{\sigma+\eps}\,\norm h_2 ,
\end{equation}
with $C_\eps$ independent of $\tilde t$, $h$ and $R$. This is the estimate we prove in Section~\ref{sec:proof}.

\section{Geometry of the evolved sets}\label{sec:geom}

Throughout this section fix $\tilde t\in[0,R^2)$ and define, for $\theta\in\Theta$,
\begin{equation}\label{eq:Psi}
\Psi_\theta(\,\cdot\,):=\gamma_R(\,\cdot\,,\tilde t,\theta):\R^n\to\R^n.
\end{equation}
Notice that $\mathcal E_y(\tilde t)=\{\Psi_\theta(y):\theta\in\Theta\}$ 
and, for each $\theta$, $\Psi_\theta$ is $C_1$-bi-Lipschitz by \eqref{C1R}.

\begin{lemma}\label{lem:homeo}
Each $\Psi_\theta$ is a homeomorphism of $\R^n$ onto $\R^n$. In particular $\Psi_\theta^{-1}$ is defined on all of $\R^n$ and is $C_1$-Lipschitz.
\end{lemma}

\begin{proof}
Write $\Psi=\Psi_\theta$ and recall $C_1^{-1}\abs{x-x'}\le\abs{\Psi(x)-\Psi(x')}\le C_1\abs{x-x'}$.

We first prove that $\Psi$ is a homeomorphism onto its image. The lower bound in \eqref{C1R} gives injectivity. The lower bound also implies the $C_1$-Lipschitz continuity of $\Psi^{-1}: \Psi(\R^n)\to \R^n$. Indeed, if $\Psi(x_k){\to}\Psi(x)$ as $k\to\infty$, then $\abs{\Psi^{-1}(\Psi(x_k))-\Psi^{-1}(\Psi(x))}=\abs{x_k-x}\le C_1\abs{\Psi(x_k)-\Psi(x)}\to0$ as $k\to\infty$. With the continuity of $\Psi$ this proves the claim.

We now check that the image is open. By the Invariance of Domain Theorem (see Corollary 19.9 in \cite{Bredon}), a continuous injective map $\R^n\to\R^n$ is open, so $\Psi(\R^n)$ is open.

We now prove that the image is closed. If $K\subset\R^n$ is compact, then $\Psi^{-1}(K)$ is the image of the bounded set $K\cap\Psi(\R^n)$ under the $C_1$-Lipschitz map $\Psi^{-1}$, therefore $\Psi^{-1}(K)$ is bounded. It is also closed, by the continuity of $\Psi$, so it is compact. Since $K$ was arbitrary, we deduce that $\Psi:\R^n\to\R^n$ is a proper map. Since a proper continuous map between locally compact Hausdorff spaces is closed (see Proposition 11.5 in \cite{Bredon}), we deduce that $\Psi(\R^n)$ is closed in $\R^n$.

We conclude. The set $\Psi(\R^n)$ is nonempty, open and closed in the connected space $\R^n$, therefore it is equal to $\R^n$.
\end{proof}

For a value $v\in\R^n$ define the map that recovers the base point from the value as
\begin{equation}\label{eq:Phi}
\Phi_v:\Theta\to\R^n,\qquad \Phi_v(\theta):=\Psi_\theta^{-1}(v),
\end{equation}
that is, \ $\Phi_v(\theta)$ is the unique $x$ with $\gamma_R(x,\tilde t,\theta)=v$ (unique by Lemma~\ref{lem:homeo}).

\begin{lemma}\label{lem:PhiLip}
Each $\Phi_v$ is Lipschitz with constant $C_1C_3R$. 
\end{lemma}

\begin{proof}
Observe that $\Psi_{\theta'}(\Phi_v(\theta'))=v=\Psi_\theta(\Phi_v(\theta))$. By this identity, the lower bound in \eqref{C1R} and \eqref{C3R},
\begin{align}
\abs{\Phi_v(\theta)-\Phi_v(\theta')}
&\ \le\ C_1\,\abs{\Psi_{\theta'}(\Phi_v(\theta))-\Psi_{\theta'}(\Phi_v(\theta'))}\label{eq:PhiLipchain}\\
&\ =\ C_1\,\abs{\Psi_{\theta'}(\Phi_v(\theta))-\Psi_{\theta}(\Phi_v(\theta))}\notag\\
&\ \le\ C_1C_3R\,\abs{\theta-\theta'}.\notag\qedhere
\end{align}
\end{proof}

\begin{lemma}[Covering numbers of $\Phi_v(\Theta)$]\label{lem:cover}
For every $\eps>0$ there exists $C_\eps>0$ such that for all $1\le r\le C_1C_3R$,
\begin{equation}\label{eq:cover}
N\bigl(\Phi_v(\Theta),r\bigr)\ \lesssim\ C_\eps\Bigl(\frac R r\Bigr)^{\beta+\eps}.
\end{equation}
\end{lemma}

\begin{proof}
Let $C_0=C_1C_3$. Since $\Phi_v$ is $(C_0R)$-Lipschitz, it takes a cover of $\Theta$ by $N(\Theta,\delta)$ balls of radius $\delta$ and sends it to a cover of $\Phi_v(\Theta)$ by sets that can each be covered by a ball of radius $C_0R\delta$. Thus, taking $\delta=\frac{r}{C_0R}\in(0,1]$,
\begin{equation}\label{eq:covpush}
N(\Phi_v(\Theta),r)\ \le\ N\Big(\Theta,\frac{r}{C_0R}\Big)
=N(\Theta,\delta)
\end{equation}
It remains to bound $N(\Theta,\delta)$ for all $\delta\in(0,1]$. By definition of $\beta$, for every $\eps>0$ there is $\delta_0=\delta_0(\eps)$ with $N(\Theta,\delta)\le\delta^{-\beta-\eps}$ for $\delta\le\delta_0$. Observe that $N(\Theta,\cdot)$ is non-increasing. Thus, when $\delta\ge\delta_0$, we have $N(\Theta,\delta)\le N(\Theta,\delta_0)$. Call $C_\eps:=\max\{1,N(\Theta,\delta_0)\}$. Then, for all $\delta\in(0,1]$,
\begin{equation}\label{eq:covtheta}
N(\Theta,\delta)\le C_\eps\,\delta^{-\beta-\eps},
\end{equation}
and \eqref{eq:cover} follows.
\end{proof}


We now turn to the two counting statements that will help us estimate the fractal density \eqref{eq:gammadef}. Tile $\R^n$ by unit cubes centered at the lattice base points $j\in\Z^n$. The first statement bounds how many of those base points can reach a single lattice unit cube in $\R^{n+1}$.

\begin{lemma}[Bounded multiplicity]\label{lem:mult}
	Fix a lattice unit cube $Q\subset\R^{n+1}$ with center $w\in\Z^{n+1}$. Let $v = \pi(w)$ and fix $c\ge1$.
	Then
	\begin{equation}\label{eq:mult}
		\#\bigl\{j\in\Z^n:\ (\mathcal E_j(\tilde t)+B^n_{c})\cap \pi(Q)\ne\emptyset\bigr\}
		\ \lesssim_{c}\ N\bigl(\Phi_v(\Theta),1\bigr)
		\ \lesssim_\eps\ R^{\beta+\eps}.
	\end{equation}
\end{lemma}

\begin{proof}
	Let $j$ be a base point from the set above and let $p\in(\mathcal E_j(\tilde t)+B^n_{c})\cap \pi(Q)$.
	Notice $\pi(Q)= \pi\big(Q^{n+1}_1(w)\big)=Q^n_1(v)$.
	Then $p\in Q_1^n(v)$, so $\abs{p-v}\le\sqrt n$, and $p\in\mathcal E_j(\tilde t)+B^n_c$, so there is $\theta\in\Theta$ with $\abs{\Psi_\theta(j)-p}<c$. Hence $\Psi_\theta(j)\in v+B^n_{c'}$ with $c':=c+\sqrt n$. Applying the Lipschitz bound for $\Psi_\theta^{-1}$ (Lemma~\ref{lem:homeo}),
	\begin{equation}\label{eq:multkey}
		\abs{j-\Phi_v(\theta)}=\bigl|\Psi_\theta^{-1}(\Psi_\theta(j))-\Psi_\theta^{-1}(v)\bigr|\ \le\ C_1\abs{\Psi_\theta(j)-v}\ \le\ C_1c' .
	\end{equation}
	So every such $j$ lies in $\Phi_v(\Theta)+B^n_{C_1c'}$. This $C_1c'$-neighborhood is covered by at most a quantity of unit balls comparable to $N(\Phi_v(\Theta),1)$, each ball containing at most $3^n$ lattice points. This proves the first inequality in \eqref{eq:mult}. Lemma~\ref{lem:cover} at $r=1$ gives the second inequality in \eqref{eq:mult}.
\end{proof}

Let us write a similar argument, at an arbitrary scale $r$, that bounds the number of base points reaching a ball of radius $r$.

\begin{lemma}[Packing]\label{lem:pack}
Let $c\ge1$ and $C\ge1$, fix $v\in\R^n$, and let $1\le r\le CR$. Then
\begin{equation}\label{eq:pack}
\#\bigl\{j\in\Z^n:\ (\mathcal E_j(\tilde t)+B^n_{c})\cap B^n_r(v)\ne\emptyset\bigr\}\ \lesssim_{\eps,c,C}\ R^{\beta+\eps}\,r^{\,n-\beta} .
\end{equation}
\end{lemma}

\begin{proof}
Apply the Lipschitz bound as in Lemma~\ref{lem:mult}. If $j$ is a base point from the set above, a point $p\in(\mathcal E_j(\tilde t)+B^n_{c})\cap B^n_r(v)$ satisfies $\abs{p-v}<r$ and $\abs{\Psi_\theta(j)-p}<c$ for some $\theta\in\Theta$. Therefore $\Psi_\theta(j)\in v+B^n_{r+c}$, hence, $j$ lies in $\Phi_{v}(\Theta)+B^n_{C_1(r+c)}$.

Write $C_0=C_1C_3$. Suppose first that $1\le r\le C_0R$. By Lemma~\ref{lem:cover}, cover $\Phi_{v}(\Theta)$ by $N(\Phi_{v}(\Theta),r)\lesssim_\eps(\frac{R}{r})^{\beta+\eps}$ balls of radius $r$ and pack at most about $r^n$ lattice points into each ball, which gives \eqref{eq:pack}. Suppose now that $C_0R\le r\le CR$, that is $r\sim R$. Since $\Phi_v$ is $(C_0R)$-Lipschitz on the bounded set $\Theta$, the set $\Phi_{v}(\Theta)$ has diameter at most about $r$, so it can be covered by a number of balls of radius $r$ only depending on the diameter of $\Theta$. Therefore $N(\Phi_{v}(\Theta),r)\lesssim 1$. The count of lattice points in each ball is at most about $r^n$. Combine this with $r^\beta\le R^{\beta+\eps}$ to obtain \eqref{eq:pack}. 
\end{proof}

\section{Proof of the fixed-time estimate}\label{sec:proof}

Fix $\tilde t$ throughout this section. We prove \eqref{eq:slice} in the range $0\le\beta\le n-1$, in which the exponent $\alpha=n-\beta$ is admissible in Proposition~\ref{prop:DZ}. The special case $n-1<\beta\le n$, where the exponent is not admissible, is treated in Section~\ref{sec:highbeta} by the same argument with $\alpha=1$. Abbreviate $\mathcal E_y=\mathcal E_y(\tilde t)$.

\subsection{Tiling}
Tile $B^n_R$ by unit cubes centered at base points $j\in\Z^n$. For $y$ in the cube centered at $j$ one has $\abs{y-j}\le\sqrt n$, so \eqref{C1R} gives $\abs{\Psi_\theta(y)-\Psi_\theta(j)}\le C_1\sqrt n$ for every $\theta$. We deduce
\begin{equation}\label{eq:EyEj}
\mathcal E_y + B^n_{C_2}\ \subset\ \mathcal E_j+B^n_{c''},
\end{equation}
where $c'':=C_2+C_1\sqrt n$. Set $E_j:=(\mathcal E_j+B^n_{c''})\times[0,R)$. Now apply Corollary~\ref{cor:lcp} with $E=E_j$. Deduce that for each $j$ there is a lattice unit cube $\mathcal Q(j)\subset\R^{n+1}$ with $  E_j \cap \mathcal Q (j)\neq \emptyset$, such that, for every $y$ in the unit cube centered at $j$,
\begin{equation}\label{eq:onecube}
\sup_{(z,t)\in(\mathcal E_y+B^n_{C_2})\times[0,R)}\abs{e^{it\Delta}h(z)}^2
\ \lesssim\ \sum_{\mathfrak l\in\Z^{n}}\frac{1}{(1+\abs{\mathfrak l})^{n+1}}\,\norm{e^{i\cdot\Delta}h_{\mathfrak l}}_{L^2(\mathcal Q(j))}^2 .
\end{equation}
Furthermore, the center of $\mathcal Q(j)$ has a projection to $\R^n$ that lies at distance at most $c''+1$ from $\mathcal E_j$, and its time coordinate lies in $[0,R)$.
By \eqref{eq:inQ}, all the cubes $\mathcal Q(j)$ with $j\in\Z^n\cap B^n_R$ satisfy
\begin{equation}\label{eq:cubesinball}
\mathcal Q(j)\ \subset\ B^{n+1}_{CR}(0),
\end{equation}
with $C:=C_2+2$.

Write $\mathcal F:=\{\mathcal Q(j):j\in\Z^n\cap B^n_R\}$ for the family of distinct cubes arising this way, with \emph{multiplicity factor}
\begin{equation}\label{eq:mudef}
\mu(Q)\ :=\ \#\{j\in\Z^n\cap B^n_R:\ \mathcal Q(j)=Q\},\qquad Q\in\mathcal F .
\end{equation}
Thus $\mu(Q)$ is the number of base points $j$ which are \emph{associated to} $Q$, that is, $\mathcal Q(j)=Q$.
Notice that a base point $j$ can be associated to a lattice unit cube $Q$ only if $E_j\cap Q\neq \emptyset$, which implies
\begin{equation}\label{eq:necessary cap}
	(\mathcal E_j+B^n_{c''})\cap\pi(Q)\neq \emptyset .
\end{equation}

Integrate \eqref{eq:onecube} in $y$ over $B^n_R$ and group equal cubes,
\begin{align}\label{eq:tiling}
\Bigl\|\sup_{(z,t)\in(\mathcal E_y+B^n_{C_2})\times[0,R)}\abs{e^{it\Delta}h(z)}\Bigr\|_{L^2(B^n_R,\,dy)}^2 
&\ 
\lesssim\ \sum_{\mathfrak l\in\Z^{n}}\frac{1}{(1+\abs{\mathfrak l})^{n+1}}\sum_{j\in\Z^n\cap B^n_R}\ \norm{e^{i\cdot\Delta}h_{\mathfrak l}}_{L^2(\mathcal Q(j))}^2
\\
&\ =\ \sum_{\mathfrak l\in\Z^{n}}\frac{1}{(1+\abs{\mathfrak l})^{n+1}}\sum_{Q\in\mathcal F}\mu(Q)\,\norm{e^{i\cdot\Delta}h_{\mathfrak l}}_{L^2(Q)}^2 .\notag
\end{align}
The family $\mathcal F$ and the multiplicity factors $\mu$ do not depend on $\mathfrak l$. Only the function $h_\mathfrak{l}$ does, and by \eqref{eq:htrans} it has the same norm and the same Fourier support as $h$. Consequently, it is enough to prove that
\begin{equation}\label{eq:target}
\sum_{Q\in\mathcal F}\mu(Q)\,\norm{e^{i\cdot\Delta}\eta}_{L^2(Q)}^2\ \le\ C_\eps\,R^{2\sigma+\eps}\,\norm \eta_2^2,
\end{equation}
for every $\eta$ with $\supp\widehat\eta\subset A(1)+B^n_{1/100}$. Indeed, applying \eqref{eq:target} to $\eta=h_{\mathfrak l}$ for each $\mathfrak l$, the sum over $\mathfrak l$ in $\eqref{eq:tiling}$ converges. Thus, the right-hand side of \eqref{eq:tiling} is bounded by $C_\eps R^{2\sigma+\eps}\norm h_2^2$. The rest of this section proves \eqref{eq:target}.

\subsection{Counting}

By Lemma~\ref{lem:mult},
\begin{equation}\label{eq:mu}
\mu(Q)\ \lesssim_\eps\ R^{\beta+\eps}
\end{equation}
for every $Q$.
 For dyadic $m$ with $1\le m\lesssim R^{\beta+\eps}$ define
\begin{equation}\label{eq:Xm}
X_m:=\bigcup_{Q:\ \mu(Q)\sim m}Q\ \subset\ B^{n+1}_{CR}(0).
\end{equation}
Let us estimate the fractal density $\nu_{X_m,\alpha}$ at $\alpha=n-\beta$. Fix a ball $B^{n+1}_r(w)\subseteq B^{n+1}_{CR}(0)$ with $r\ge1$. Each cube of $X_m$ inside it is associated to about $m$ base points by definition. Every base point associated to a cube inside $B^{n+1}_r(w)$ is counted by Lemma~\ref{lem:pack} at the value $v=\pi(w)$ and radius $r$. Hence
\begin{align}
\frac{\#\{Q\subset B^{n+1}_r(w)\cap X_m\}}{r^{\,n-\beta}}
&\ \sim\ \frac{\#\{j:\ \mathcal Q(j)\subset B^{n+1}_r(w)\cap X_m\}}{m\;r^{\,n-\beta}}\label{eq:density}\\
\intertext{which is, by \eqref{eq:necessary cap}}
&\leq\frac{\#\bigl\{j:\ (\mathcal E_j(\tilde t)+B^n_{c''})\cap B^n_r(v)\cap\pi(X_m)\ne\emptyset\bigr\}}{m\;r^{\,n-\beta}}
\intertext{and now, by Lemma~\ref{lem:pack},}
&\ \lesssim_\eps\ \frac{R^{\beta+\eps}\,r^{\,n-\beta}}{m\;r^{\,n-\beta}}
\ =\ \frac{R^{\beta+\eps}}{m},\notag
\intertext{so that}
\ \nu_{X_m,\,n-\beta}&\lesssim_\eps\frac{R^{\beta+\eps}}{m}.
\label{eq:density2}
\end{align}
Since $0\le\beta\le n-1$, we may apply Proposition~\ref{prop:DZ} with $\alpha:=n-\beta\in[1,n]$ and $CR$ in place of $R$
 to arrive at
\begin{equation}\label{eq:perclass}
\norm{e^{i\cdot\Delta}\eta}_{L^2(X_m)}^2\ \lesssim_\eps\ \Bigl(\frac{R^{\beta}}{m}\Bigr)^{\frac2{n+1}}R^{\frac{n-\beta}{n+1}+\eps}\norm \eta_2^2 .
\end{equation}

\subsection{Summing}
For each unit cube $Q$ there is a unique $m$ such that $Q\subset X_m$. Using this, \eqref{eq:mu} and \eqref{eq:perclass},
\begin{align}
\sum_{Q\in\mathcal F}\mu(Q)\,\norm{e^{i\cdot\Delta}\eta}_{L^2(Q)}^2
&\ \sim\ 
\sum_{m\sim1}^{R^{\beta+\eps}}
m\,\norm{e^{i\cdot\Delta}\eta}_{L^2(X_m)}^2\label{eq:summation}\\
&\ \lesssim_\eps\ 
\sum_{m\sim1}^{R^{\beta+\eps}}
m^{1-\frac{2}{n+1}}\,R^{\frac{2\beta}{n+1}}\,R^{\frac{n-\beta}{n+1}}\,R^{\eps}\norm \eta_2^2
\ =\ 
\sum_{m\sim1}^{R^{\beta+\eps}}
m^{\frac{n-1}{n+1}}\,R^{\frac{n+\beta}{n+1}+\eps}\norm \eta_2^2\notag\\
\intertext{and, since $n\geq 1$, by \eqref{eq:mu},
}
&\qquad\ \lesssim_\eps\ \bigl(R^{\beta}\bigr)^{\frac{n-1}{n+1}}R^{\frac{n+\beta}{n+1}+2\eps}\norm \eta_2^2
\ =\ R^{\frac{n(\beta+1)}{n+1}+2\eps}\norm \eta_2^2 .\notag
\end{align}
This is \eqref{eq:target} with $\sigma=\frac{n(\beta+1)}{2(n+1)}$, after renaming $\eps$. By \eqref{eq:tiling}, and taking square roots, we obtain \eqref{eq:slice} with that value of $\sigma$, uniformly in $\tilde t$. Together with the reductions of Section~\ref{sec:red}, this proves Theorem~\ref{thm:main} in the range $0\le\beta\le n-1$.

\section{The special case $n-1<\beta\le n$}\label{sec:highbeta}

Fix $\tilde t$ throughout this section. When $\beta>n-1$ we have $\alpha=n-\beta<1$ and Proposition~\ref{prop:DZ} no longer applies for that choice of $\alpha$. We keep the whole argument and only change the choice of $\alpha$, taking the smallest admissible value $\alpha=1$. The only step that changes is the estimate of the fractal density, which is now taken at $\alpha=1$, so that the denominator is $r$ rather than $r^{\,n-\beta}$. By Lemma~\ref{lem:pack}, the number of base points reaching an $r$-ball is at most about $R^{\beta+\eps}r^{\,n-\beta}$. Thus,
\begin{align}
	\frac{\#\{Q\subset B^{n+1}_r(w)\cap X_m\}}{r}
	&\ \sim\ \frac{\#\{j:\ \mathcal Q(j)\subset B^{n+1}_r(w)\cap X_m\}}{m\;r}\label{eq:densityhigh}\\
	&\ \lesssim_\eps\ \frac{R^{\beta+\eps}\,r^{\,n-\beta-1}}{m}
	\ \leq\ \frac{R^{\beta+\eps}}{m},\notag
\end{align}
where the last inequality uses $r\ge1$ and $n-\beta-1<0$, so that the largest value is attained at $r=1$.
Therefore,
\begin{equation}\label{eq:density1}
\nu_{X_m,1}\ \lesssim_\eps\ \frac{R^{\beta+\eps}}{m}.
\end{equation}
Then Proposition~\ref{prop:DZ} with $\alpha=1$ gives $\norm{e^{i\cdot\Delta}\eta}_{L^2(X_m)}^2\lesssim_\eps(\frac{R^\beta}{m})^{\frac2{n+1}}R^{\frac1{n+1}+\eps}\norm \eta_2^2$, and the summation is done as before
\begin{align}
\sum_{Q\in\mathcal F}\mu(Q)\,\norm{e^{i\cdot\Delta}\eta}_{L^2(Q)}^2
&\ \lesssim_\eps\ \sum_{m\sim1}^{R^{\beta+\eps}} m^{\frac{n-1}{n+1}}R^{\frac{2\beta}{n+1}}R^{\frac1{n+1}+\eps}\norm \eta_2^2\label{eq:summation1}\\
&\ \lesssim_\eps\ \bigl(R^\beta\bigr)^{\frac{n-1}{n+1}}R^{\frac{2\beta+1}{n+1}+2\eps}\norm \eta_2^2\notag\\
&\ 
=\ R^{\,\beta+\frac1{n+1}+2\eps}\norm \eta_2^2 .\notag
\end{align}
This is \eqref{eq:target} with $\sigma=\frac\beta2+\frac1{2(n+1)}$. By \eqref{eq:tiling}, and taking square roots, we obtain \eqref{eq:slice}, which proves Theorem~\ref{thm:main} in the range $n-1<\beta\le n-\tfrac{1}{n+1}$. The value is larger than $\frac n2$ in $n-\frac1{n+1}<\beta\leq n$, so we keep $\sigma=\frac n2$ from the Sobolev embedding in this range.

\section*{Acknowledgments}

The author was partially supported by Grant PID2022-142202NB-I00 funded by \linebreak
AEI/10.13039/501100011033.
The author would like to thank Fernando Soria and Ana Vargas for helpful discussions on this work. The author has no conflict of interest to declare.

\section*{AI declaration}
Anthropic's Claude Opus 5 model was used during the writing phase of this article. Specifically, the model was used to transcribe handwritten notes, polish the language, format references, and find \cite{DKWZ} in search of the known ranges of \eqref{eq:fixedparam} in the literature. All mathematical results have been found and verified by the author.
\pagebreak

\bigskip
\footnotesize\raggedright
\noindent\textsc{Javier Minguillón, Department of Mathematics, Universidad Autónoma de Madrid, 28049 Madrid, Spain}\\
\textit{E-mail address:} \texttt{javier.minguillon@uam.es}

\end{document}